\documentclass{amsart}

\usepackage[backend=biber,maxbibnames=5,maxalphanames=5,style=alphabetic,bibencoding=utf8,giveninits,url=false,isbn=false]{biblatex}
\AtBeginBibliography{\small}

\usepackage{verse}

\newlength\lineindent

\usepackage[all,hyperref]{bi-discrete}
\usepackage{cleveref}
\usepackage{tikz}
\usetikzlibrary{calc}

\usepackage{mathtools}

\usepackage{algorithm}
\usepackage{algpseudocode}

\providecommand{\du}{\, \mathrm{d}u}
\providecommand{\ds}{\, \mathrm{d}s}

\providecommand{\R}{\mathbb{R}}

\providecommand{\cA}{\mathcal{A}}

\usepackage{dsfont}
\usepackage{pgfplots}
\usepackage{tikz}
\usepackage{listings}
\usepackage{stmaryrd}

\usepackage{tikz}
\usepackage{pgfplots}
\usepackage{pgfplotstable} 
\pgfplotsset{compat=1.18} 

\pgfplotsset{
  width=.65\linewidth,
  axis background/.style={fill=black!5!white},
  grid style={densely dotted,semithick},
  legend style={
    legend columns=1,
    legend pos=outer north east
  },
  compat=newest 
}
\pgfplotscreateplotcyclelist{MyCol1}{%
    {red,mark = *,every mark/.append style={solid,scale=0.4,fill=red}},
	{dashed,blue,mark = square,every mark/.append style={solid,scale=0.4,fill=blue}},
    {dotted,black,mark = x,every mark/.append style={solid,scale=0.8,fill=black}}
    }
\pgfplotscreateplotcyclelist{MyCol1b}{%
    {red},
	{dashed,blue},
    {dotted,black}
    }

\pgfplotscreateplotcyclelist{MyColors}{%
    {red,mark = *,every mark/.append style={solid,scale=0.4,fill=red}},
	{dotted,red,mark = *,every mark/.append style={solid,scale=0.4,fill=red}},
    {teal,mark = x,every mark/.append style={solid,scale=0.4,fill=teal}},
    {blue,mark = square*,every mark/.append style={solid,scale=0.4,fill=blue}},
    {dotted,teal,mark = x,every mark/.append style={solid,scale=0.4,fill=teal}},
    {dotted,blue,mark = square*,every mark/.append style={solid,scale=0.4,fill=blue}},
{dash dot,red,mark = *,every mark/.append style={solid,scale=0.4,fill=red}},
    {dash dot,teal,mark = x,every mark/.append style={solid,scale=0.4,fill=teal}},
    {dash dot,blue,mark = square*,every mark/.append style={solid,scale=0.4,fill=blue}},}

\pgfplotscreateplotcyclelist{MyColors2}{%
    {blue,mark = square,every mark/.append style={solid,scale=0.3,fill=blue}},
    {red,mark = *,every mark/.append style={solid,scale=0.3,fill=red}},
	{dotted,blue,mark = square*,every mark/.append style={solid,scale=0.1,fill=blue}},
    {dotted,red,mark = *,every mark/.append style={solid,scale=0.1,fill=red}},
	{teal,mark = x,every mark/.append style={solid,scale=0.3,fill=teal}},
    {dotted,teal,mark = x,every mark/.append style={solid,scale=0.1,fill=teal}},
    {dotted,blue,mark = square*,every mark/.append style={solid,scale=0.1,fill=blue}},
{dash dot,red,mark = *,every mark/.append style={solid,scale=0.4,fill=red}},
    {dash dot,teal,mark = x,every mark/.append style={solid,scale=0.4,fill=teal}},
    {dash dot,blue,mark = square*,every mark/.append style={solid,scale=0.4,fill=blue}},}

\pgfplotscreateplotcyclelist{MyColors3}{%
    {blue,mark = square,every mark/.append style={solid,scale=0.5,fill=blue}},
    {red,mark = *,every mark/.append style={solid,scale=0.5,fill=red}},
	{teal,mark = x,every mark/.append style={solid,scale=0.5,fill=teal}},
    {dotted,teal,mark = x,every mark/.append style={solid,scale=0.1,fill=teal}},
    {dotted,blue,mark = square*,every mark/.append style={solid,scale=0.1,fill=blue}},
{dash dot,red,mark = *,every mark/.append style={solid,scale=0.4,fill=red}},
    {dash dot,teal,mark = x,every mark/.append style={solid,scale=0.4,fill=teal}},
    {dash dot,blue,mark = square*,every mark/.append style={solid,scale=0.4,fill=blue}},}

\makeatletter
\@namedef{subjclassname@2020}{%
  \textup{2020} Mathematics Subject Classification}
\makeatother

\usepackage{lineno}

\begin{document}
\lstset{language=Python,
basicstyle=\small, 
keywordstyle=\color{black}\bfseries, 
commentstyle=\color{blue}, 
stringstyle=\ttfamily, 
showstringspaces=false,
numbers=left, 
numberstyle=\small, 
numbersep=10 pt,
xleftmargin= 27pt,
}

\author[J.~Storn]{Johannes Storn}
\address[J.~Storn]{Faculty of Mathematics \& Computer Science, Institute of Mathematics, Leipzig University, Augustusplatz 10, 04109 Leipzig, Germany}
\email{jstorn@math.uni-bielefeld.de}


\thanks{The work of the author was supported by the Deutsche Forschungsgemeinschaft (DFG, German Research Foundation) -- SFB 1283/2 2026 -- 317210226.}

\title[Dual IRLS scheme for graph $p$-Laplacians and $\ell^p$ regression]{The dual IRLS scheme for (hyper-)graph $p$-Laplacians and $\ell^p$ regression with large exponents} 

\begin{abstract}
We introduce an iterative scheme for discrete convex minimization problems of $p$-Laplace type such as variational graph $p$-Laplace problems and $\ell^p$ regression. In each iteration, the scheme solves only a weighted least-squares problem. We verify linear convergence for suitably regularized problems and derive convergence to any prescribed tolerance.
\end{abstract}
\maketitle

\section{Introduction}
The graph Laplacian is a fundamental operator in data science and network analysis, with important applications in semi-supervised learning \cite{ZhuGhahramaniLafferty03}, spectral clustering and graph partitioning \cite{Luxburg07}, image segmentation \cite{ShiMalik00}, manifold learning \cite{BelkinNiyogi03}, graph signal processing \cite{ShumanNarangFrossardOrtega13}, and diffusion or consensus processes on networks \cite{OlfatiFaxMurray07}.
These applications rely on the graph Laplacian as a discrete notion of smoothness and diffusion, whose spectral structure also captures global geometric features of the underlying graph.
In certain situations, however, the standard graph Laplacian has severe drawbacks. 
For example, in semi-supervised learning with very few labels, Laplacian learning may become degenerate \cite{NadlerSrebroZhou09}. This motivates the use of the (variational) graph $p$-Laplacian with large exponents $p\gg 2$ as a nonlinear alternative with improved regularity and smoothing properties \cite{AlamgirLuxburg11,FloresCalderLerman22}.
A practical downside of this approach is that the associated minimization problem is difficult to solve. More precisely, let $B\colon \mathbb{R}^N \to \mathbb{R}^M$ for $N,M\in \mathbb{N}$ be a discrete gradient and let  $(w_\alpha)_{\alpha=1}^M \subset (0,\infty)$ denote positive weights. We seek the (constrained) minimizer of the $p$-Dirichlet energy
\begin{align*}
J(v) \coloneqq \frac{1}{p} \sum_{\alpha=1}^M w_\alpha\, |(B v)_\alpha|^p \qquad\text{for all }v\in \mathbb{R}^N.
\end{align*}  
A closely related minimization problem is $\ell^p$ regression, where one seeks with given matrix $A \in \mathbb{R}^{M \times N}$ and vector $b\in \mathbb{R}^M$ the (constrained) minimizer of the residual 
\begin{align*}
G(v) \coloneqq \lVert A v - b \lVert^p_{\ell^p}\qquad\text{for all }v\in \mathbb{R}^N.
\end{align*}
Such problems occur for example in digital filters~\cite{VargasBurrus09}, $L^p$ polynomial regression~\cite{MeyerMuscoMuscoWoodruffZhou23}, or clustering~\cite{HuangVishnoi20,HuangShaofengJianingXuan22}. Corresponding solvers have been studied intensively in recent years \cite{BubeckCohenLeeLi18,AdilPengSachdeva19,AdilKyngPengSachdeva24,EneNguyenVladu25}
While there exist schemes that work well for $p < 2$ \cite{Osborne85,Li93}, existing iterative schemes struggle for large exponents $p \gg 2$. 
We address this challenging problem by exploiting the iteratively reweighted least-squares (IRLS) scheme. This scheme, however, is known to diverge for values $p> 3$, see for example \cite{Osborne85}. We thus extend an idea that has been introduced for PDEs in \cite{BalciDieningStorn23}: Rather than minimizing the primal energy, we exploit duality to obtain an iterative scheme for the dual problem. The corresponding dual minimization scheme involves the dual exponent $p'$ which is smaller than two for values $p$ larger than two. Consequently, the IRLS converges linearly for the dual problem, which we verify for suitable regularizations following arguments from \cite{DieningFornasierTomasiWank20}. We then rewrite the linearized dual problem as a reweighted primal problem, leading to a novel reweighted primal least-squares scheme, which we call the dual reweighted least-squares method (dIRLS). 

Section~\ref{sec:graph-plap} introduces our abstract model problem and illustrates its generality by some examples. We proceed in Section~\ref{sec:dIRLS} by introducing the dual formulation of our minimization problem as well as the dual IRLS scheme. Our analysis includes a linear convergence result for suitable regularizations as well as an investigation of the impact of the regularization. Combining these results yields the convergence statements in \eqref{eq:COnvergenceResult}.
We conclude our studies by numerical experiments indicating the practical performance of our scheme as well as its superiority to existing methods.

\section{The (Hyper-)Graph $p$-Laplacian}\label{sec:graph-plap}
Throughout this paper, we consider the following model situation covering various graph, hypergraph, and $\ell^p$ regression models.
Examples are discussed in Section~\ref{subsec:Applications}.

\subsection{Minimization problem}
Let $\cA=\lbrace 1,\dots,M\rbrace$ with $M\in \mathbb{N}$ be an abstract index set representing, for example, edges of a graph, vertex-hyperedge incidences of a hypergraph, or hyperarcs. Moreover, let $V = \lbrace 1,\dots,N\rbrace$ with $N\in \mathbb{N}$ be a set of indices representing for example vertices and let $B\colon \mathbb{R}^N \to \mathbb{R}^M$ be a linear operator, representing a discrete gradient. Moreover, let $\phi \colon [0,\infty) \to [0,\infty)$ be an N-function, defined below, such as $\phi(t) = t^p/p$ with $p \in (1,\infty)$.
\begin{definition}[N-function]\label{def:NfunctionIntro}
A function $\phi\colon \mathbb{R}_{\geq 0} \to \mathbb{R}_{\geq 0}$ is an N-function, if
\begin{enumerate}
\item $\phi$ is continuous and convex,
\item there is a right-continuous and non-decreasing function $\phi'\colon \mathbb{R}_{\geq 0} \to \mathbb{R}_{\geq 0}$ that satisfies $\lim_{t\to \infty} \phi'(t) = \infty$ as well as $\phi'(0) = 0$ such that we have the positivity $\phi'(t) >0$ for all $t>0$ and 
\begin{align*}
\phi(t) = \int_0^t \phi'(s)\ds.
\end{align*}
\end{enumerate}
\end{definition} 
We aim at minimizing an energy under a linear constraint. The constraint is defined via a matrix $C \in \mathbb{R}^{m_c \times N}$ and a vector $d\in \mathbb{R}^{m_c}$ with $m_c \in \mathbb{N}_0$. In particular, we want to minimize the energy over all vectors $v\in \mathbb{R}^N$ with $Cv = d$. This constraint may include identities such as $v_i = d_i$ for all $i \in L \subset \lbrace 1,\dots, N\rbrace$, where $L$ represents the set of labeled indices. We suppose that there exists a $g \in \mathbb{R}^N$ that satisfies the constraint; that is, $C g = d$ and define the space $ V_0 \coloneqq \ker(C)$.

We write $(Bv)_\alpha$ for the $\alpha$-th component with $\alpha \in \cA$ and $v\in \mathbb{R}^N$. 
Given positive weights $w = (w_\alpha)_{\alpha \in \cA} \in \mathbb{R}^M$ and a right-hand side $f \in \mathbb{R}^N$, we set the energy
\begin{align}\label{eq:AbstractEnergy}
J(v) \coloneqq \sum_{\alpha \in \cA} w_\alpha \phi(|(Bv)_\alpha|) - f \cdot v\qquad\text{for all }v\in \mathbb{R}^N.
\end{align} 
With this energy, we seek the (constrained) minimizer $u_g = u + g$ with 
\begin{align}\label{eq:MinProb}
u = \argmin_{v\in V_0} J(v + g).
\end{align}
If we have $\ker(B) \cap V_0 = \lbrace 0 \rbrace$, the direct method in the calculus of variations yields well-posedness of the minimization problem above in the sense that there exists a unique minimizer. To characterize this minimizer $u\in V_0$, we define the function
\begin{align*}
A_\phi(t) \coloneqq \begin{cases}
 \tfrac{\phi'(|t|)}{|t|} t &\text{for }t \neq 0,\\
0 & \text{for }t = 0.
\end{cases}
\end{align*}
Differentiation characterizes $u\in V_0$ via the variational problem
\begin{align}\label{eq:var-abstract}
\sum_{\alpha\in\cA} w_\alpha\, A_\phi\big((B(u+g))_\alpha\big)\,(Bv)_\alpha
= f\cdot v
\qquad\text{for all } v\in V_0.
\end{align}

\subsection{Applications}\label{subsec:Applications}
Before we introduce our iterative solver for the problem in \eqref{eq:MinProb}, we discuss typical applications fitting in the abstract framework.
\subsubsection{Graph-based semi-supervised learning}\label{subsec:app-ssl}
Let $G=(V,W)$ be a weighted graph with $N\coloneqq |V|$ vertices $V$ and nonnegative edge weights $W = (w_{ij})_{i,j = 1}^N \in \mathbb{R}^{N\times N}$, where $w_{ij} \approx 1$ if vertex $i$ is similar to vertex $j$ and $w_{ij} \approx 0$ if $i$ and $j$ are dissimilar. The corresponding $p$-Dirichlet energy with $p \in (1,\infty)$ reads
\begin{align*}
J(v) \coloneqq \frac{1}{p} \sum_{i,j \in V} w_{ij} |v_i-v_j|^p \qquad\text{for all }v \in \mathbb{R}^N.
\end{align*}
To fit into our framework, we assume that the graph is undirected; that is, $w_{ij} = w_{ji}$ for all $i,j\in V$. We define the set of all edges $\mathcal{E} \coloneqq \big\lbrace \lbrace i,j\rbrace \colon i,j\in V, i\neq j\big\rbrace$ and set the weights $w_e \coloneqq 2 w_{ij} = 2 w_{ji}$ for all edges $e = \lbrace i , j\rbrace \in \mathcal{E}$. The  index set reads  
\begin{align*}
\cA \coloneqq \big\lbrace e \in \mathcal{E} \colon w_e > 0\rbrace \eqsim \lbrace 1, \dots, M\rbrace\qquad\text{with }M \coloneqq |\cA|. 
\end{align*}
For any edge $e = \lbrace i,j\rbrace \in \mathcal{E}$, we fix its orientation $e = [i,j]$ and define the discrete gradient $(B v)_e \coloneqq v_i - v_j$ for all $v\in \mathbb{R}^N$. This allows us to rewrite the $p$-Dirichlet energy in \eqref{eq:AbstractEnergy} with $\phi(t) \coloneqq t^p/p$ as 
\begin{align*}
J(v) = \sum_{e\in \cA} w_e \phi\big( |(Bv)_e| \big)  \qquad\text{for all }v \in \mathbb{R}^N.
\end{align*}
This model is the standard variational graph $p$-Laplacian, introduced as a regularizer in \cite{AlamgirLuxburg11} and further investigated, among others, in \cite{BridleZhu13,AlaouiChengRamdasWainwrithJordan16,Flores18,SlepcevThorpe19,FloresCalderLerman22}. 
The approach generalizes popular graph Laplacian regularizers which correspond to the (linear) case $p=2$ and have been introduced in \cite{ZhuGhahramaniLafferty03} for graph-based semi-supervised learning. The overall goal in this application is to learn labels indexed by $U = \lbrace 1,\dots,n\rbrace$ on a large set of data points $V$ when only a small set indexed by $L = \lbrace n+1,\dots,N\rbrace$ with $N-n \ll n$ in $V$ is labeled by given values $(g_i)_{i \in L} \in \mathbb{R}^{N-n}$; that is, our linear constraint reads $v_i = g_i$ for all $i \in L$  with kernel $V_0 = \lbrace v \in \mathbb{R}^N \colon v_i = 0\text{ for all }i \in L\rbrace  \eqsim \mathbb{R}^n$.
The key idea is to exploit the structure of the unlabeled data: points that are ``close'' should get similar predictions. This approach then leads to the minimization problem in \eqref{eq:MinProb}. The linear setting $p=2$ becomes ill-posed as the number of vertices tends to infinity while the number of labeled data points remains finite \cite{NadlerSrebroZhou09}, causing the minimizer $u$ to be nearly constant throughout the entire graph with spikes near the labeled data in practical applications. Large exponents $p \gg 2$ provide a remedy to this problem, see for example \cite{FloresCalderLerman22}.
\subsubsection{Hypergraph models}\label{subsec:app-hypergraph}
Hypergraphs extend graph models with vertices $V$ by allowing for hyperedges $h\subset V$ connecting more than two vertices in order to model more complex relationships between objects \cite{ZhouHuangSchoelkopf06}. Let $\mathcal{H}$ denote the set of hyperedges and let $(w_h)_{h\in \mathcal{H}} \subset (0,\infty)$ denote associated positive weights. 
There are several non-equivalent hypergraph Laplacians in the literature -- not all of them fit the framework with linear gradient $B$. The class of operators fitting into the framework include for example the clique expansion suggested in \cite{ZienSchlagChan02} that creates a graph with edges 
\begin{align*}
\mathcal{E} \coloneqq \big\lbrace \lbrace i,j\rbrace \colon i,j \in h \in \mathcal{H}\text{ with }i\neq j\big\rbrace.
\end{align*}
The corresponding edge weights read for all $e = \lbrace i,j\rbrace \in \mathcal{E} $ \cite[Eq.~11]{AgarwalBransonBelongie06}
\begin{align*}
w_e \coloneqq \argmin_{x > 0} \sum_{ \lbrace h \in \mathcal{H}\colon i,j\in h\rbrace} ( x - w_h)^2.
\end{align*} 
This framework fits the representation discussed in Section~\ref{subsec:app-ssl} above. Further suitable models are discussed for example in \cite{AgarwalBransonBelongie06}.

\subsubsection{Linear regression in $\ell^p$}\label{subsec:linReg}
A closely related instance of our abstract energy minimization problem is $\ell^p$-regression. Given $p \in [2,\infty)$, $A\in\mathbb{R}^{M\times N}$, $b\in\mathbb{R}^M$,  $\tilde C\in\mathbb{R}^{m_c\times N}$, and $\tilde d\in\mathbb{R}^{m_c}$ with $m_c \in \mathbb{N}_0$, this problem seeks the minimizer
\begin{align*}
\tilde u_g \in \argmin_{\lbrace v \in \mathbb{R}^N\colon \tilde Cv=\tilde d \rbrace} \lVert Av-b\rVert_p^p.
\end{align*}
The mapping $v \mapsto  Av-b$ is affine. In order to obtain a linear operator that fits into our framework, we set $\tilde N \coloneqq N + M$ and define with identity matrix $\identity \in \mathbb{R}^{M \times M}$
\begin{align*}
B \coloneqq (A \ \ -\identity) \in \mathbb{R}^{M \times \tilde N}.
\end{align*}
Moreover, we set 
\begin{align*}
C \coloneqq \begin{pmatrix}
\tilde{C} & 0 \\ 0 & \identity
\end{pmatrix} \in \mathbb{R}^{(m_c + M) \times\tilde  N}
\qquad\text{and}\qquad
d \coloneqq \begin{pmatrix}
\tilde d \\ 
b
\end{pmatrix} \in \mathbb{R}^{m_c + M}.
\end{align*}
With $V_0 \coloneqq \ker C$, $g \in \mathbb{R}^{\tilde N}$ with $C g = d$, and $\phi(t) \coloneqq t^p/p$ we obtain the minimization problem: Seek $u_g$ via $u_g = u +  g$ with
\begin{align*}
u \in \argmin_{ v \in V_0 }  \sum_{\alpha = 1}^M  \phi \big( | (B(v  + g))_\alpha |\big).
\end{align*}
The minimizers of the original and lifted problem satisfy by definition 
\begin{align*}
u_g = \begin{pmatrix}
\tilde{u}_g \\ b
\end{pmatrix}.
\end{align*}

\section{Dual IRLS scheme}\label{sec:dIRLS}

A prominent approach for solving variational problems such as \eqref{eq:var-abstract} is the iteratively reweighted least squares (IRLS) linearization, also known as \Kacanov{} iteration in the context of PDEs \cite{DieningFornasierTomasiWank20,HeidWihler20}. This approach computes iteratively the solution $u^{n+1} \in V_0$  to the linearized problem 
\begin{align*}
\sum_{\alpha\in\cA} w_\alpha \frac{\phi'(|(B(u^n+g))_\alpha|)}{|(B(u^n+g))_\alpha|} (B(u^{n+1}+g))_\alpha\,(Bv)_\alpha = f\cdot v \qquad\text{for all } v\in V_0.
\end{align*}
However, this approach fails for (regularized versions of) $\phi(t) = t^p/p$ with exponents $p \geq 3$, as shown in \cite[Rem.~21]{DieningFornasierTomasiWank20} with an example in the context of the PDE and as observed numerically in \cite{FloresCalderLerman22} for the graph $p$-Laplacian. Our ansatz exploits a remedy used for PDEs in \cite{BalciDieningStorn23,DieningStorn25}: We apply the IRLS linearization to the dual problem with dual exponent $p' = p/(p-1) < 2$ for $p>2$. This allows for the direct application of the linear convergence result in \cite{DieningFornasierTomasiWank20}. Then we rewrite the linearized dual problem as a weighted primal problem, resulting in an iteratively reweighted least squares method with alternative weights. The remainder of this section introduces and analyzes this idea in more detail.

\subsection{Dual problem}
The derivation of the dual problem follows classical theory on convex analysis, cf.~\cite{RockafellarWets98}.
In the following, we discuss some of the involved steps, as these intermediate results are used in the analysis of the dIRLS.

We set the Fenchel conjugate  $\phi^*$ for all $r \geq 0$ by 
\begin{align*}
\phi^*(r) \coloneqq \sup_{s\geq 0} \big(rs - \phi(s)\big). 
\end{align*} 
Figure~\ref{fig:NfunctionsIntro} displays the convex conjugate $\phi^*$ and illustrates the following properties proven for example in \cite{DieningEttwein08,DieningFornasierTomasiWank20}. 
\begin{figure}
  \begin{tikzpicture}[scale=1.6]
    \fill [gray!90, domain=0:2.6, variable=\x]
      (0, 0)
      -- plot ({\x}, {1-cos(deg(pi/3*\x)))})
      -- (2.6, 0)
      -- cycle;
      \fill [gray!30, domain=0:2, variable=\x]
      (0, 0)
      -- plot ({\x}, {1-cos(deg(pi/3*\x)))})
      -- (0, {1-cos(deg(pi/3*2))})
      -- cycle;

	\node[fill=white] () at (2,.7) {$\phi(s)$};
	\node[fill=white] () at (.8,1) {$\phi^*(r)$};
	\node[below] () at (2.6,0) {$s$};
	\node[left] () at (0,{1-cos(deg(pi/3*2))}) {$r$};

	\draw[dashed] (0,0) 
	-- (2.6,0)
	-- (2.6, {1-cos(deg(pi/3*2))}) 
	-- (0, {1-cos(deg(pi/3*2))});
	
    \draw [thick] [->] (0,0)--(2.8,0) node[right] {$t$};
    \draw [thick] [->] (0,0)--(0,2.1) node[above] {$u$};
    \draw [domain=0:2.8, variable=\x]
      plot ({\x}, {1-cos(deg(pi/3*\x)))}) node[right] {$\phi'(t),\ (\phi^*)'(u)$};
  \end{tikzpicture}
  \caption{Plot of derivatives $\phi'$ and $(\phi^*)'$. The area marked dark gray equals $\phi(s)$, the area marked light gray equals $\phi^*(r)$. The area surrounded by the dashed line equals $rs$.}\label{fig:NfunctionsIntro}
\end{figure}
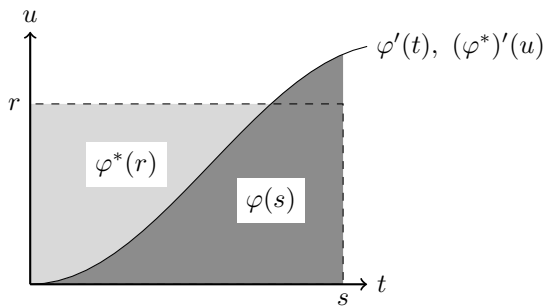%
\begin{proposition}[Conjugate of an N-function]\label{prop:Conjugate}
Let $\phi\colon \mathbb{R}_{\geq 0} \to \mathbb{R}_{\geq 0}$ be an N-function and define the right-continuous inverse
\begin{equation*}
(\phi')^{-1}(t) \coloneqq \sup\lbrace r \geq 0 \mid \phi'(r) \leq t\rbrace\qquad\text{for all }t\geq 0.
\end{equation*}
Then we have the following.
\begin{enumerate}
\item If $\phi'$ is strictly increasing, the function $(\phi')^{-1}$ equals the inverse of $\phi'$.\label{itm:inverse}
\item The convex conjugate of $\phi$ is an $N$-function with representation \label{itm:DualDerivativeEqual}
\begin{equation*}
\phi^*(r) =  \int_0^r (\phi')^{-1}(u)\du\qquad\text{for all }r \geq 0.
\end{equation*} 
\item One has for all $t\geq 0$ the identity \label{itm:ConjugateC}
\begin{equation*}
\phi^*(\phi'(t)) = \phi'(t)t - \phi(t).
\end{equation*}
\item The function $\phi$ equals its bi-conjugate, that is for all $t\geq 0$\label{itm:biConjugate}
\begin{equation*}
\phi(t) = \phi^{**}(t) \coloneqq \sup_{s \geq 0} \big(t s - \phi^*(s)\big).
\end{equation*}
\end{enumerate}
\end{proposition}
We set the dual energy
\begin{align*}
J^*(\tau) \coloneqq \sum_{\alpha\in\cA} w_\alpha\,\phi^*\big(w_\alpha^{-1}|\tau_\alpha|\big) - \tau^\top B g \qquad\text{for all }\tau \in \mathbb{R}^M.
\end{align*}
Moreover, the dual space reads
\begin{align*}
\Sigma_f & \coloneqq  \lbrace \tau\in\R^M\colon \tau^\top B v =f \cdot v\ \text{for all } v\in V_0\rbrace.
\end{align*}
We set the minimizer $u_g = u + g \in \mathbb{R}^N$ and $\sigma \in \mathbb{R}^M$ with 
\begin{align}\label{eq:Minimizers}
u = \argmin_{v \in V_0} J(v + g)\qquad \text{and}\qquad \sigma = \argmin_{\tau \in \Sigma_f } J^*(\tau).
\end{align} 
We define the function 
\begin{align}\label{eq:DefFenchelConjugate}
A_\phi^* (t) \coloneqq \begin{cases}
\tfrac{(\phi^*)'(|t|)}{|t|}t&\text{for }t \neq 0,\\
0 & \text{for }t = 0.
\end{cases}
\end{align}
\begin{theorem}[Duality]\label{thm:Duality}
Let $\phi$ be an N-function and suppose that $\ker(B) \cap V_0 = \lbrace 0 \rbrace$. Then we have the identity
\begin{align*}
J(u_g) + f\cdot g = \min_{v \in V_0} J(v + g)+ f\cdot g = - \min_{\tau \in \Sigma_f} J^*(\tau) = - J^*(\sigma).
\end{align*}
Moreover, the minimizers $u_g$ and $\sigma$ satisfy for all $\alpha \in \cA$
\begin{align*}
\sigma_\alpha &= w_\alpha\,A_\phi\big((Bu_g)_\alpha\big)\qquad\text{and}\qquad
(Bu_g)_\alpha = A_\phi^*\big(w_\alpha^{-1}\sigma_\alpha\big).
\end{align*}
\end{theorem}
\begin{proof}
This is a classical result. A proof that fits well to our notation can be found in \cite[Thm.~2.3]{DieningStorn25}.
\end{proof}
We set the space 
\begin{align*}
\Sigma_0 \coloneqq \lbrace \tau\in\mathbb{R}^M\colon \tau^\top B v =0 \text{ for all } v \in V_0 \rbrace.
\end{align*}
Differentiating $J^*$ leads to the equivalent characterization of the minimizer $\sigma \in \Sigma_f$ as solution to the variational problem
\begin{align*}
\sum_{\alpha\in\cA} A_\phi^* \big(w_\alpha^{-1} \sigma_\alpha\big) \tau_\alpha &= \tau^\top Bg \qquad \text{for all }\tau\in\Sigma_0.
\end{align*} 
Using a Lagrange multiplier $\tilde u \in V_0$, we can rewrite the problem as the saddle-point system: Seek $\sigma \in \mathbb{R}^M$ and $\tilde u \in V_0$ such that
\begin{align}\label{eq:dualSaddle}
\begin{aligned}
\sum_{\alpha\in\cA} \big( A_\phi^* \big(w_\alpha^{-1} \sigma_\alpha\big) \tau_\alpha - (B(\tilde u+g))_\alpha \tau_\alpha \big) &=0 &&\text{for all }\tau\in\R^M,\\
\sigma^\top Bv &= f\cdot v &&\text{for all }v\in V_0.
\end{aligned}
\end{align}
Testing with the canonical basis vectors $\tau = e_\alpha \in \mathbb{R}^M$ yields for all $\alpha \in \cA$ the pointwise identity
\begin{align*}
A_\phi^*\big(w_\alpha^{-1}\sigma_\alpha\big) = (B(\tilde u+g))_\alpha.
\end{align*}
Hence, Theorem~\ref{thm:Duality} shows that the Lagrange multiplier equals the primal solution in the sense that $u_g = \tilde u + g$ and $u = \tilde u$.

\subsection{Dual IRLS}
We compute the solution to the saddle-point problem in \eqref{eq:dualSaddle} iteratively for each $n\in \mathbb{N}_0$ via the IRLS; that is, given $\sigma^{n} \in \mathbb{R}^M$, we seek the solution $\sigma^{n+1}\in \mathbb{R}^M$ and $u^{n+1}\in V_0$ to
\begin{align}\label{eq:dualKacSaddle}
\begin{aligned}
\sum_{\alpha\in\cA}
\frac{(\phi^*)'(w_\alpha^{-1}|\sigma^n_\alpha|)}{|\sigma^n_\alpha|} \sigma^{n+1}_\alpha \tau_\alpha
-\tau^\top B(u^{n+1}+g) &=0 &&\text{for all }\tau\in\mathbb{R}^M,\\
(\sigma^{n+1})^\top Bv &= f\cdot v &&\text{for all }v\in V_0.
\end{aligned}
\end{align}
Testing again pointwise leads for all $\alpha \in \cA$ to the identity 
\begin{align}\label{eq:aeadssafg}
\frac{(\phi^*)'(w_\alpha^{-1} |\sigma^{n}_\alpha|)}{|\sigma^{n}_\alpha|} \sigma^{n+1}_\alpha  =  (B(u^{n+1}+g))_\alpha. 
\end{align}
The weight $(\phi^*)'(w_\alpha^{-1} t)/t$ is positive for all $t > 0$. For the limiting case $t = 0$ we interpret the weight as the right-limit in the sense that 
\begin{align}\label{eq:Def4Zero}
\frac{(\phi^*)'(w_\alpha^{-1} 0)}{0} \coloneqq \lim_{t \searrow 0}\frac{(\phi^*)'(w_\alpha^{-1} t)}{t}. 
\end{align} 
Let us assume that $(\phi^*)'(w_\alpha^{-1} t)/t >0$ for all $t\geq 0$, which we will ensure later by suitable assumptions, see Theorem~\ref{thm:ConvOfDualIRLS} below. We can reformulate the identity in \eqref{eq:aeadssafg}, leading to
\begin{align}\label{eq:IDenitSigmaNone}
\sigma^{n+1}_\alpha = (B(u^{n+1}+g))_\alpha|\sigma^{n}_\alpha| /(\phi^*)'(w_\alpha^{-1} |\sigma^{n}_\alpha|).
\end{align}
Applying this identity to the second equation of \eqref{eq:dualKacSaddle} results with 
\begin{align*}
a_\alpha \coloneqq |\sigma^{n}_\alpha| /(\phi^*)'(w_\alpha^{-1} |\sigma^{n}_\alpha|) \qquad\text{for all }\alpha \in \cA
\end{align*} 
in the weighted least-squares problem: Seek $u^{n+1} \in V_0$ with 
\begin{align}\label{eq:DualKac}
\sum_{\alpha \in \cA} a_\alpha (B(u^{n+1}+g))_\alpha (Bv)_\alpha &= f\cdot v \qquad\text{for all }v\in V_0.
\end{align}
Having solved this equation, we obtain $\sigma^{n+1} \in \Sigma_f$ via the formula in \eqref{eq:IDenitSigmaNone}. 
This allows us to rewrite the dIRLS scheme in \eqref{eq:dualKacSaddle} in terms of the weighted primal least-squares problem \eqref{eq:DualKac}. 
\begin{remark}[Built-in error control]\label{rem:BuiltInErrorControl}
Since $\sigma^n \in \Sigma_f$ and $u^n \in \mathbb{R}^N$, the duality result in Theorem~\ref{thm:Duality} yields the computable guaranteed upper energy error bound
\begin{align*}
\max\lbrace J(u^n + g) - J(u_g), J^*(\sigma^n) - J^*(\sigma)\rbrace \leq J(u^n + g) + J^*(\sigma^n) + f\cdot g.
\end{align*}
\end{remark}
It remains to verify convergence of the dIRLS. We therefore adapt arguments from \cite{DieningFornasierTomasiWank20}, which exploit uniform convexity.
\begin{definition}[Uniform convexity]\label{def:UniformConvexity}
We call an N-function $\phi$ uniformly convex, if
\begin{align*}
\frac{\phi'(t)-\phi'(s)}{t-s} \eqsim \frac{\phi'(s)}{s}\qquad\text{for all }0 \leq t < s.
\end{align*}
\end{definition}
A slightly stronger assumption, which is for example the uniform convexity condition in \cite{DieningEttwein08}, is the following.
\begin{lemma}[Alternative characterization]
Suppose that the N-function $\phi$ is piecewise twice differentiable in the sense that $\phi \in W^{2,\infty}_\textup{loc}((0,\infty))$. Moreover, assume we have with constant $0<c\leq C < \infty$ the bound 
\begin{align}\label{eq:AltCharUnifConv}
c \leq \frac{\phi''(t)t}{\phi'(t)}  \leq C < \infty\qquad\text{for almost all }t\geq 0.
\end{align}
Then $\phi$ is uniformly convex. If in addition $1 \leq c$, we have superquadratic growth
\begin{align*}
\frac{\phi'(t)}{t} \leq \frac{\phi'(T)}{T}\qquad\text{for all }0< t \leq T.
\end{align*}
\end{lemma}
\begin{proof}
The first statement of the lemma is proven in \cite[Lem.~30]{DieningFornasierTomasiWank20}. To verify the superquadratic growth, we set
\begin{align*}
\psi(t)\coloneqq \frac{\phi'(t)}{t} \qquad\text{for all }t>0.
\end{align*}
Since $\phi\in W^{2,\infty}_{\mathrm{loc}}((0,\infty))$, we have $\phi'\in W^{1,\infty}_{\mathrm{loc}}((0,\infty))$ and hence $\psi\in W^{1,\infty}_{\mathrm{loc}}((0,\infty))$. We have for almost all $t>0$ the identity
\begin{align}\label{eq:tempisadaas}
\psi'(t) = \frac{\phi''(t)t-\phi'(t)}{t^2} = \frac{\phi'(t)}{t^2}
\left(\frac{\phi''(t)t}{\phi'(t)}-1\right).
\end{align}
If $1 \leq c$ in \eqref{eq:AltCharUnifConv}, the combination of \eqref{eq:tempisadaas} and $\phi'(s)>0$ for $s>0$ yields
\begin{align*}
0 \leq \psi'(t) \qquad\text{for almost all }t>0.
\end{align*}
Thus $\psi$ is non-decreasing on $(0,\infty)$, proving the claim.
\end{proof}
We can now state linear convergence of the dIRLS scheme. Note that the quadratic growth condition will be obtained for general integrands such as $\phi(t) = t^p/p$ via a regularization discussed in Section~\ref{subsec:Regularization} below. Moreover, in practical applications it is not important to use an initial iterate $\sigma^0 \in \Sigma_f$ -- any unconstrained $\sigma^0 \in \mathbb{R}^M$ can be used since the first iterate satisfies $\sigma^1 \in \Sigma_f$. The theorem below then applies with an index shift.
\begin{theorem}[Convergence of the dIRLS scheme]\label{thm:ConvOfDualIRLS}
Suppose that $\ker(B)\cap V_0=\lbrace 0\rbrace$.
Let $\phi$ be a uniformly convex N-function such that $\phi'(t)/t$ is non-decreasing; that is,
\begin{align}\label{eq:NonDecreasing}
\frac{\phi'(t)}{t} \leq \frac{\phi'(T)}{T}\qquad\text{for all }0 < t \leq T.
\end{align}
Moreover, we assume boundedness in the sense that
\begin{align}\label{eq:QuadGrowth}
0 < c_\phi \coloneqq \lim_{t \searrow 0} \frac{\phi'(t)}{t} \leq \lim_{T \to \infty} \frac{\phi'(T)}{T} \eqqcolon C_\phi  < \infty.
\end{align}
Finally, we assume that $\sigma^0 \in \Sigma_f$. Then the dIRLS scheme \eqref{eq:DualKac} with \eqref{eq:Def4Zero} is well-defined and converges linearly in the sense that the iterates $\sigma^n \in \Sigma_f$ with $n\in \mathbb{N}_0$ satisfy with some contraction factor $\kappa \coloneqq C_\textup{uc} c_\phi/ C_\phi$ with constant $C_\textup{uc} < \infty$ depending solely on the uniform convexity the estimate
\begin{align*}
J^*(\sigma^n) - J^*(\sigma) \leq (1-\kappa)^n\big(J^*(\sigma^0)-J^*(\sigma)\big).
\end{align*}
\end{theorem}
Before we verify the theorem, let us equivalently reformulate the assumptions in terms of the convex conjugate $\phi^*$. 
\begin{lemma}[Equivalent assumptions]\label{lem:equiAssumtpion}
Let $\phi$ be an N-function with convex conjugate $\phi^*$. Then the monotonicity condition in \eqref{eq:NonDecreasing} is satisfied if and only if $(\phi^*)'(t)/t$ is non-increasing; that is, 
\begin{align}\label{eq:non-decreasing}
\frac{(\phi^*)'(t)}{t} \geq \frac{(\phi^*)'(T)}{T}\qquad\text{for all }0< t \leq T.
\end{align}
Moreover, the boundedness in \eqref{eq:QuadGrowth} is satisfied if and only if additionally
\begin{align*}
0 < C_\phi^{-1} = \lim_{T \to \infty} \frac{(\phi^*)'(T)}{T} \leq \lim_{t \searrow 0} \frac{(\phi^*)'(t)}{t} = c_\phi^{-1} < \infty.
\end{align*}
\end{lemma}
\begin{proof}
\textit{Step 1 (Monotonicity).}
Suppose that \eqref{eq:NonDecreasing} is satisfied.
Let $0 < t\leq T$ and set $s\coloneqq \phi'(t)$ and $S\coloneqq\phi'(T)$. Since $\phi'$ is non-decreasing (Definition~\ref{def:NfunctionIntro}) we obtain $0<s\leq S$. Moreover, we can equivalently rewrite
\begin{align*}
\frac{\phi'(t)}{t} \leq \frac{\phi'(T)}{T} \qquad\text{ as }\qquad \frac{t}{s}\geq \frac{T}{S}.
\end{align*}
Since Proposition~\ref{prop:Conjugate}~\ref{itm:DualDerivativeEqual} yields the identity $(\phi^*)' = (\phi')^{-1}$, where $(\phi')^{-1}$ is the inverse of $\phi'$ according to Proposition~\ref{prop:Conjugate}~\ref{itm:inverse} with strict monotonicity of $\phi'$ following by \eqref{eq:NonDecreasing},  we obtain $t = (\phi')^{-1}(s) =  (\phi^*)'(s)$ and $T=(\phi')^{-1}(S) = (\phi^*)'(S)$. This leads to the claimed non-increasing property
\begin{align*}
\frac{(\phi^*)'(s)}{s} \geq \frac{(\phi^*)'(S)}{S}.
\end{align*}
The reverse implication follows by the same steps.

\textit{Step 2 (Boundedness).}
Assume first that \eqref{eq:QuadGrowth} holds.
Since $t \mapsto \phi'(t)/t$ is non-decreasing by \eqref{eq:NonDecreasing}, we obtain $c_\phi \leq \phi'(t)/t \leq C_\phi$ and, equivalently, 
\begin{align*}
C_\phi^{-1} \phi'(t) \leq t \leq c_\phi^{-1} \phi'(t) \qquad\text{for all }t>0.
\end{align*}
By Step~1, the monotonicity condition \eqref{eq:NonDecreasing} implies that $\phi'$ is strictly increasing on $(0,\infty)$. Hence $(\phi')^{-1}$ is the actual inverse of $\phi'$, and Proposition~\ref{prop:Conjugate}~\ref{itm:DualDerivativeEqual} yields $(\phi^*)'=(\phi')^{-1}$. 
Let $s>0$ and $r\coloneqq (\phi^*)'(s)=(\phi')^{-1}(s)$.
We have $s=\phi'(r)$ and 
\begin{align*}
C_\phi^{-1}s=C_\phi^{-1}\phi'(r)\leq r\leq c_\phi^{-1}\phi'(r)=c_\phi^{-1}s.
\end{align*}
Dividing by $s>0$, we arrive at
\begin{align}\label{eq:asdsagggg}
C_\phi^{-1}\leq \frac{(\phi^*)'(s)}{s}\leq c_\phi^{-1}
\qquad\text{for all }s>0.
\end{align}
Since, by Step~1, the function $s\mapsto (\phi^*)'(s)/s$ is non-increasing, the one-sided limits at $0$ and $\infty$ exist. Moreover, if $t \searrow 0$ and $s \coloneqq \phi'(t)$, then $s \searrow 0$ and
\begin{align*}
\frac{(\phi^*)'(s)}{s} = \frac{(\phi')^{-1}(s)}{s} = \frac{t}{\phi'(t)} = \left(\frac{\phi'(t)}{t}\right)^{-1}
\to c_\phi^{-1}.
\end{align*}
This yields
\begin{align*}
\lim_{s\searrow 0}\frac{(\phi^*)'(s)}{s}=c_\phi^{-1}.
\end{align*}
Similarly, we obtain 
\begin{align*}
\lim_{s\to\infty}\frac{(\phi^*)'(s)}{s}=C_\phi^{-1}.
\end{align*}
This shows one implication. The opposite implication follows by the same arguments due to the property $\phi^{**} = \phi$.
\end{proof}
The condition in \eqref{eq:non-decreasing} is crucial, since it enables the following observation involving the relaxed energy defined for all $\tau,\chi \in \mathbb{R}^M$ as
\begin{align}\label{eq:DefRelaxedJ}
\begin{aligned}
J^*(\tau,\chi) &\coloneqq \sum_{\alpha \in \cA} \Big( \frac{1}{2}  \frac{(\phi^*)'(w_\alpha^{-1}|\chi_\alpha|)}{|\chi_\alpha|} |\tau_\alpha|^2 - |\chi_\alpha|\frac{(\phi^*)'(w_\alpha^{-1} |\chi_\alpha|)}{2} \\
&\qquad\qquad +w_\alpha \phi^*(w_\alpha^{-1}|\chi_\alpha|)  \Big) - \tau^\top Bg.
\end{aligned}
\end{align}
\begin{lemma}[Weighted quadratic energy]\label{lem:weigthedQuadEnergy}
Suppose $\phi^*$ satisfies the monotonicity condition in \eqref{eq:non-decreasing}. 
Then we have the minimization property
\begin{align*}
J^*(\tau) = J^*(\tau,\tau) \leq J^*(\tau,\chi)\qquad\text{for all }\tau,\chi \in \mathbb{R}^M.
\end{align*}
\end{lemma}
\begin{proof}
Set the function $\psi(t)\coloneqq (\phi^*)'(t)/t$ for all $t > 0$, which is non-increasing on $(0,\infty)$ by the assumption in \eqref{eq:non-decreasing}. We fix $t>0$ and define 
\begin{align*}
F_t(s)\coloneqq \frac12 \psi(t)s^2-\phi^*(s)\qquad\text{for all }s > 0.
\end{align*}
Since $(\phi^*)'$ increases monotonically by Proposition~\ref{prop:Conjugate}~\ref{itm:DualDerivativeEqual} and Definition~\ref{def:NfunctionIntro}, it is almost everywhere differentiable (Lebesgue's theorem on differentiability of monotone functions). Thus, differentiation yields for almost all $s>0$ that
\begin{align*}
F_t'(s)= \psi(t)s-(\phi^*)'(s) = s\big(\psi(t)-\psi(s)\big).
\end{align*}
Since $\psi$ is non-increasing, it follows that $F_t'(s)  \geq 0$ for $s > t$ and $F_t'(s) \leq 0$ for $s < t$.
Hence $F_t$ attains its minimum at $s=t$, and therefore
\begin{align*}
F_t(t) \leq F_t(s)\qquad\text{for all }s\geq 0.
\end{align*}
This shows that
\begin{align*}
\frac12\frac{(\phi^*)'(t)}{t} t^2-\phi^*(t)\leq \frac12 \frac{(\phi^*)'(t)}{t} s^2-
\phi^*(s)\qquad\text{for all }s,t\geq 0.
\end{align*}
We fix $\alpha\in\cA$ and apply this inequality with $s\coloneqq w_\alpha^{-1}|\tau_\alpha|$ and $t\coloneqq w_\alpha^{-1}|\chi_\alpha|$. Multiplying by $w_\alpha$ yields
\begin{align*}
w_\alpha \phi^*(w_\alpha^{-1}|\tau_\alpha|) &\leq
\frac12 w_\alpha
\frac{(\phi^*)'(w_\alpha^{-1}|\chi_\alpha|)}{w_\alpha^{-1}|\chi_\alpha|}
\bigl(w_\alpha^{-1}|\tau_\alpha|\bigr)^2 \\
&\qquad -\frac12\frac{(\phi^*)'(w_\alpha^{-1}|\chi_\alpha|)}{w_\alpha^{-1}|\chi_\alpha|} (w_\alpha^{-1}|\chi_\alpha|)^2
+w_\alpha \phi^*(w_\alpha^{-1}|\chi_\alpha|) \\
&= \frac12 \frac{(\phi^*)'(w_\alpha^{-1}|\chi_\alpha|)}{|\chi_\alpha|} |\tau_\alpha|^2
-|\chi_\alpha|\frac{(\phi^*)'(w_\alpha^{-1}|\chi_\alpha|)}{2} + w_\alpha \phi^*(w_\alpha^{-1}|\chi_\alpha|).
\end{align*}
Summing over all $\alpha\in\cA$ proves
\begin{align*}
J^*(\tau)\le J^*(\tau,\chi).
\end{align*}
The identity $J^*(\tau ) = J^*(\tau,\tau)$ follows by definition.
\end{proof}
A further auxiliary result is the following ``correct'' notion of distance that relies on the uniform convexity of $\phi$ and the function
\begin{align}\label{eq:DefV}
V_{\phi^*}(t) \coloneqq \begin{cases}
\sqrt{\frac{(\phi^*)'(|t|)}{|t|}}t&\text{for }t \neq 0,\\
0&\text{for }t = 0.
\end{cases}
\end{align}
\begin{lemma}[Natural distance]\label{lem:natDist}
Suppose that $\phi$ is uniformly convex.
The minimizer $\sigma \in \Sigma_f$ defined in \eqref{eq:Minimizers} and any $\tau \in \Sigma_f$ satisfy
\begin{align*}
J^*(\tau) - J^*(\sigma) \leq \sum_{\alpha \in \cA} \big( A^*_\phi(w_\alpha^{-1} \tau_\alpha ) - A^*_\phi(w_\alpha^{-1} \sigma_\alpha) \big)  (\tau_\alpha - \sigma_\alpha)  \lesssim J^*(\tau) - J^*(\sigma).
\end{align*}
The pointwise values are for all $\alpha \in \cA$ proportional to 
\begin{align*}
\big( A^*_\phi(w_\alpha^{-1} \tau_\alpha ) - A^*_\phi(w_\alpha^{-1} \sigma_\alpha) \big)  (\tau_\alpha - \sigma_\alpha) &\eqsim \frac{(\phi^*)'\big(w_\alpha^{-1} (|\sigma_\alpha| \vee |\tau_\alpha|)\big)}{|\sigma_\alpha| \vee |\tau_\alpha|} |\sigma_\alpha - \tau_\alpha|^2\\
&\eqsim w_\alpha\, |V_{\phi^*}(w_\alpha^{-1} \tau_\alpha) - V_{\phi^*}(w_\alpha^{-1} \sigma_\alpha)|^2.
\end{align*}
The hidden constants in the estimates depend solely on the hidden constants in Definition~\ref{def:UniformConvexity}. 
\end{lemma}
\begin{proof}
By the convexity of $J^*$ and the property $(J^*)'(\sigma) = 0$ in $\Sigma_0^*$ we obtain
\begin{align*}
&J^*(\tau) - J^*(\sigma) \leq  (J^*)'(\tau)[\tau - \sigma] = \big( (J^*)'(\tau) - (J^*)'(\sigma)\big) [\tau - \sigma] \\
& = \sum_{\alpha \in \cA} \big( A^*_\phi(w_\alpha^{-1} \tau_\alpha ) - A^*_\phi(w_\alpha^{-1} \sigma_\alpha) \big) (\tau_\alpha - \sigma_\alpha)\qquad\qquad\text{for all }\tau \in \Sigma_f.
\end{align*}
The second  bound in the lemma relies on properties of (shifted) N-functions, displayed for example in the appendix of \cite{DieningFornasierTomasiWank20}. A corresponding proof that extends directly to the discrete setting can be found in \cite[Lem.~42]{DieningFornasierTomasiWank20}. The pointwise equivalence can be found in \cite[Lem.~3 and 41]{DieningFornasierTomasiWank20}.
\end{proof}

With these preliminary results, we verify the convergence of the dIRLS scheme.
\begin{proof}[Proof of Theorem~\ref{thm:ConvOfDualIRLS}]
The convergence proof follows arguments from \cite{BalciDieningStorn23}, which is itself based on \cite{DieningFornasierTomasiWank20}. Let $\sigma \in \Sigma_f$ denote the minimizer of the dual energy $J^*$ and let $\sigma^n\in \Sigma_f$ denote the current iterate of the dIRLS scheme in \eqref{eq:dualKacSaddle}.
Since $\sigma^n - \sigma \in \Sigma_0$, the definition of $\sigma^{n+1} \in \Sigma_f$ in \eqref{eq:dualKacSaddle}, Lemma~\ref{lem:natDist}, and the Cauchy--Schwarz inequality yield
\begin{align}\label{eq:addends}
\begin{aligned}
&J^*(\sigma^n) - J^*(\sigma) \leq \sum_{\alpha \in \cA} \big(A_\phi^*(w_\alpha^{-1}\sigma^n_\alpha) - A_\phi^*(w_\alpha^{-1}\sigma_\alpha)\big) (\sigma^n_\alpha -\sigma_\alpha)\\
&  = \sum_{\alpha \in \cA} \frac{(\phi^*)'(w_\alpha^{-1}|\sigma^n_\alpha|)}{|\sigma^n_\alpha|} (\sigma^n_\alpha - \sigma^{n+1}_\alpha)\cdot (\sigma^n_\alpha - \sigma_\alpha)\\
& \leq \frac{1}{2\gamma} \sum_{\alpha \in \cA}  \frac{(\phi^*)'(w_\alpha^{-1}|\sigma^n_\alpha|)}{|\sigma^n_\alpha|} |\sigma^n_\alpha-\sigma^{n+1}_\alpha|^2 + \frac{\gamma}{2} \sum_{\alpha \in \cA} \frac{(\phi^*)'(w_\alpha^{-1}|\sigma^n_\alpha|)}{|\sigma^n_\alpha|} |\sigma^n_\alpha-\sigma_\alpha|^2.
\end{aligned}
\end{align}
It follows by \eqref{eq:dualKacSaddle} that
\begin{align*}
\sum_{\alpha\in \cA}  \frac{(\phi^*)'(w_\alpha^{-1}|\sigma^n_\alpha|)}{|\sigma^n_\alpha|} \sigma_\alpha^{n+1}(\sigma_\alpha^{n+1} - \sigma_\alpha^n)  = (\sigma^{n+1} - \sigma^n)^\top Bg.
\end{align*}
This identity, the definition of the relaxed energy in \eqref{eq:DefRelaxedJ}, and Lemma~\ref{lem:weigthedQuadEnergy} yield
\begin{align}\label{eq:aaddesns2}
\begin{aligned}
& \sum_{\alpha \in \cA} \frac{1}{2} \frac{(\phi^*)'(w_\alpha^{-1}|\sigma^n_\alpha|)}{|\sigma^n_\alpha|} |\sigma^n_\alpha-\sigma^{n+1}_\alpha|^2 =\frac{1}{2} \sum_{\alpha \in \cA}  \frac{(\phi^*)'(w_\alpha^{-1}|\sigma^n_\alpha|)}{|\sigma^n_\alpha|} |\sigma^n_\alpha|^2  - (\sigma^n)^\top Bg \\
&\quad - \frac12 \sum_{\alpha \in \cA}  \frac{(\phi^*)'(w_\alpha^{-1}|\sigma^n_\alpha|)}{|\sigma^n_\alpha|} |\sigma^{n+1}_\alpha|^2 + (\sigma^{n+1})^\top Bg \\
&= J^*(\sigma^n,\sigma^n) - J^*(\sigma^{n+1},\sigma^n) \leq J^*(\sigma^n) - J^*(\sigma^{n+1}).
\end{aligned}
\end{align}
To bound the second addend in \eqref{eq:addends}, we exploit the assumption in \eqref{eq:QuadGrowth}, which yields due to Lemma~\ref{lem:equiAssumtpion} the bound
\begin{align*}
\frac{(\phi^*)'(t)}{t} \leq c_\phi^{-1} \leq \frac{C_\phi}{c_\phi} \frac{(\phi^*)'(s)}{s}\qquad\text{for all }t,s \geq 0.
\end{align*}
This inequality and Lemma~\ref{lem:natDist} yield with some constant $C_\textup{uc} < \infty$ depending solely on the uniform convexity
\begin{align}\label{eq:Dasdsada}
\begin{aligned}
\sum_{\alpha \in \cA} \frac{(\phi^*)'(w_\alpha^{-1}|\sigma^n_\alpha|)}{|\sigma^n_\alpha|} |\sigma^n_\alpha-\sigma_\alpha|^2 & \leq   \frac{C_\phi}{c_\phi} \sum_{\alpha \in \cA} \frac{(\phi^*)'\big(w_\alpha^{-1}(|\sigma^n_\alpha| \vee |\sigma_\alpha|)\big)}{|\sigma^n_\alpha|\vee |\sigma_\alpha|} |\sigma^n_\alpha-\sigma_\alpha|^2 \\
 &\leq 2  C_\textup{uc} \frac{C_\phi}{c_\phi} \big( J^*(\sigma^n) - J^*(\sigma)\big).
\end{aligned}
\end{align}
Hence, exploiting the bounds for the addends in \eqref{eq:addends} yields for all $\gamma >0$ that
\begin{align*}
\gamma(1-\gamma C_\textup{uc} C_\phi/c_\phi )\big(J^*(\sigma^n) - J^*(\sigma)\big) \leq J^*(\sigma^n) - J^*(\sigma^{n+1}).
\end{align*}
This inequality and $\max_{\gamma>0} \gamma(1- \gamma C_\textup{uc} C_\phi/c_\phi ) = c_\phi/(4 C_\textup{uc} C_\phi)$ yield the energy reduction
\begin{align*}
c_\phi/(4 C_\textup{uc} C_\phi)\big(J^*(\sigma^n) - J^*(\sigma)\big) \leq J^*(\sigma^n) - J^*(\sigma^{n+1}).
\end{align*}
Rearranging the terms yields with $\kappa \coloneqq c_\phi/(4 C_\textup{uc} C_\phi)$
\begin{align*}
J^*(\sigma^{n+1}) - J^*(\sigma) & = \big(J^*(\sigma^{n}) - J^*(\sigma)\big) - \big( J^*(\sigma^{n}) -J^*(\sigma^{n+1}) \big) \\
&\leq (1-\kappa) \big(J^*(\sigma^n) - J^*(\sigma)\big).
\end{align*}
Applying this estimate inductively concludes the proof.
\end{proof}
We conclude our convergence analysis with an investigation of the primal approximation $B(u^{n+1}_g)$ with $u^{n+1}_g\coloneqq u^{n+1}+g$ and Lagrange multiplier $u^{n+1}\in \mathbb{R}^N$ defined in \eqref{eq:dualKacSaddle}. Since the natural notions of distance discussed in Lemma~\ref{lem:natDist} extend to the primal integrand $\phi$, we focus on the distance with respect to $V_\phi$ defined as in \eqref{eq:DefV}; that is,
\begin{align*}
V_{\phi}(t) \coloneqq \begin{cases}
\sqrt{\frac{\phi'(|t|)}{|t|}}t&\text{for }t\neq 0,\\
0&\text{for }t=0.
\end{cases}
\end{align*}
 The following theorem shows that the primal approximation converges with the same rate with respect to this natural norm.
\begin{theorem}[Convergence of the primal variable]\label{thm:ConvPrimal}
Suppose that $\ker(B)\cap V_0=\lbrace 0\rbrace$.
Let $z_\alpha^n \coloneqq A_\phi^*(w^{-1}_\alpha\sigma_\alpha^n)$ denote the reconstruction of our dual approximation $\sigma^n\in \Sigma_f$ and suppose that $\phi$ is uniformly convex and satisfies \eqref{eq:NonDecreasing}. Then we have
\begin{align*}
&\sum_{\alpha\in \cA} w_\alpha |V_\phi((Bu_g)_\alpha) - V_\phi (z^n_\alpha)|^2 + \sum_{\alpha \in \cA} w_\alpha \frac{\phi'(|z_\alpha^n|)}{|z_\alpha^n|}|z^n_\alpha - B(u^{n+1}_g)|^2\\
&\qquad\qquad\qquad\qquad\qquad\qquad\quad\qquad\qquad\qquad\qquad\qquad \lesssim J^*(\sigma^n) - J^*(\sigma).
\end{align*}
If the integrand $\phi$ satisfies \eqref{eq:QuadGrowth}, the addends in the second sum are controlled by
\begin{align*}
\frac{c_\phi}{C_\phi} | V_\phi \big(z^n_\alpha\big) - V_\phi((Bu^{n+1}_g)_\alpha)|^2 &\lesssim \frac{\phi'(|z_\alpha^n|)}{|z_\alpha^n|}|z^n_\alpha - B(u^{n+1}_g)|^2\\
& \lesssim \frac{C_\phi}{c_\phi} | V_\phi \big(z_\alpha^n\big) - V_\phi((B u^{n+1}_g)_\alpha)|^2.
\end{align*}
Combining these statements yields
\begin{align*}
J(u_g^{n+1}) - J(u_g) \lesssim \frac{C_\phi}{c_\phi} \big( J^*(\sigma^n) - J^*(\sigma)\big).
\end{align*}
The hidden constants solely depend on the uniform convexity of $\phi$.
\end{theorem} 
\begin{proof}
Let $\sigma\in \Sigma_f$ denote the exact dual minimizer \eqref{eq:Minimizers} and set componentwise the vector $z_\alpha^n\coloneqq A_\phi^*(w^{-1}_\alpha\sigma_\alpha^n) \in \mathbb{R}^M$. Theorem~\ref{thm:Duality}, Proposition~\ref{prop:Conjugate}~\ref{itm:inverse}, and Lemma~\ref{lem:natDist} yield
\begin{align*}
\sum_{\alpha\in \cA} w_\alpha |V_\phi((Bu_g)_\alpha) - V_\phi(z^n_\alpha)|^2 & = \sum_{\alpha\in \cA} w_\alpha |V_{\phi^*}(w_\alpha^{-1}\sigma_\alpha) - V_{\phi^*}(w_\alpha^{-1}\sigma^n_\alpha)|^2\\
& \eqsim J^*(\sigma^n) - J^*(\sigma).
\end{align*}
To bound the second term in the theorem's first bound, we define for all $\alpha \in \cA$
\begin{align*}
\tilde z_\alpha^{n+1} \coloneqq B(u^{n+1}_g)_\alpha = \frac{(\phi^*)'(w_\alpha^{-1}|\sigma_\alpha^n|)}{|\sigma_\alpha^n|} \sigma_\alpha^{n+1}. 
\end{align*}
The difference to $z^n_\alpha \coloneqq A_\phi^*(w^{-1}_\alpha\sigma_\alpha^n)$ equals for all $\alpha \in \cA$
\begin{align*}
\tilde z^{n+1}_\alpha - z^n_\alpha = \frac{(\phi^*)'(w_\alpha^{-1}| \sigma_\alpha^n|)}{|\sigma_\alpha^n|} (\sigma^{n+1}_\alpha -  \sigma^n_\alpha).
\end{align*} 
With the weight $a_\alpha \coloneqq |\sigma_\alpha^n|/(\phi^*)'(w_\alpha^{-1} |\sigma_\alpha^n|)$ we obtain by squaring the identity and summing over all $\alpha \in \cA$ the equality
\begin{align*}
\sum_{\alpha \in \cA} 
 a_\alpha\, |\tilde z^{n+1}_\alpha - z^n_\alpha|^2 =\sum_{\alpha \in \cA} 
 \frac{(\phi^*)'(w_\alpha^{-1} |\sigma_\alpha^n|)}{|\sigma_\alpha^n|} |\sigma^{n+1}_\alpha -  \sigma^n_\alpha|^2.
\end{align*}
The right-hand side equals the term in \eqref{eq:aaddesns2} and can consequently be controlled by the same upper bound; that is,
\begin{align*}
\sum_{\alpha \in \cA}  a_\alpha\, |\tilde z^{n+1}_\alpha - z^n_\alpha|^2 \leq 2\big(J^*(\sigma^n) - J^*(\sigma^{n+1})\big) \leq 2 \big(J^*(\sigma^n) - J^*(\sigma)\big). 
\end{align*}
Proposition~\ref{prop:Conjugate}~\ref{itm:inverse} and the definition of $z^n$ show 
\begin{align*}
a_\alpha \coloneqq \frac{|\sigma_\alpha^n|}{(\phi^*)'(w_\alpha^{-1} |\sigma_\alpha^n|)} = w_\alpha \frac{\phi'(|z_\alpha^n|)}{|z^n_\alpha|}\qquad\text{for all }\alpha\in \cA.
\end{align*}
Combining these results verifies the first statement of the theorem. The second statement follows by an argument similar to \eqref{eq:Dasdsada}. The last statement results from the first two statements and the triangle inequality.
\end{proof}
\subsection{Regularization}\label{subsec:Regularization}
Our primary goal is the minimization of the $p$-Dirichlet energy; that is, the integrand reads $\phi(t) = t^p/p$ with exponent $p > 2$. While this integrand matches the non-decreasing condition \eqref{eq:NonDecreasing}, the required boundedness in \eqref{eq:QuadGrowth} fails in the sense that $\phi'(t)/t = t^{p-2}$ satisfies
\begin{align*}
0 =  \lim_{t\searrow 0} \frac{\phi'(t)}{t} \leq \lim_{t\to \infty} \frac{\phi'(t)}{t} = \infty.
\end{align*}
We could directly regularize the dual functional, as for example done in \cite{BalciDieningStorn23}. However, we exploit the regularization of the primal energy introduced in \cite{DieningFornasierTomasiWank20} for $p<2$. The regularization involves a relaxation interval $\delta = [\delta_-,\delta_+]\subset (0,\infty)$ with lower and upper bounds $0<\delta_-\leq \delta_+< \infty$. Given an N-function $\phi$ satisfying the non-decreasing condition \eqref{eq:NonDecreasing}, we define its (shifted) regularization $\phi_\delta$ by
\begin{align*}
\phi_\delta(t) \coloneqq 
\begin{cases}
\frac{\phi'(\delta_-)}{2\delta_-} t^2+ \phi(\delta_-)-\frac{\delta_-}{2}\phi'(\delta_-)
&\text{for }  t\in [0, \delta_-),\\
\phi(t) &\text{for } t\in [\delta_-, \delta_+],\\
\frac{\phi'(\delta_+)}{2\delta_+} t^2 + \phi(\delta_+)-\frac{\delta_+}{2}\phi'(\delta_+)
&\text{for }  t\in (\delta_+,\infty).
\end{cases}
\end{align*}
This energy is shifted in the sense that $\phi_\delta(0) = \phi(\delta_-)-\frac{\delta_-}{2}\phi'(\delta_-) \neq 0$. 
The derivative of $\phi_\delta$, which is independent of the shift, reads
\begin{align*}
\phi_\delta'(t) = t \left( \frac{\phi'(\delta_-)}{\delta_-} \vee \frac{\phi'(t)}{t} \wedge \frac{\phi'(\delta_+)}{\delta_+}\right)\qquad\text{for all }t\geq 0.
\end{align*}
\begin{remark}[$p$-Dirichlet energy]
For the $p$-Dirichlet energy $\phi(t) = t^p/p$ this regularization, motivated in \cite{DieningFornasierTomasiWank20} via a bi-level optimization, clips the derivative between the values $\delta_-$ and $\delta_+$ in the sense that 
\begin{align*}
\phi_\delta'(t) = t(\delta_- \vee t \wedge \delta_+)^{p-2}\qquad\text{for all }t\geq 0.
\end{align*}
\end{remark}
Using the identity $(\phi_\delta^*)' = (\phi_\delta')^{-1}$ in Proposition~\ref{prop:Conjugate}, we can characterize the derivative of the Fenchel conjugate by
\begin{align*}
(\phi_\delta^*)'(r)=
\begin{cases}
\frac{\delta_-}{\phi'(\delta_-)}r
&\text{for } r\in [0,\phi'(\delta_-)],\\
(\phi^*)'(r)
&\text{for }  r\in (\phi'(\delta_-), \phi'(\delta_+)],\\
\frac{\delta_+}{\phi'(\delta_+)}r
&\text{for } r\in (\phi'(\delta_+),\infty).
\end{cases}
\end{align*}
Integrating this derivative leads to the (shifted) dual integrand $\phi_\delta^*(r) = \int_0^r (\phi^*_\delta)'(s)\ds - \phi(\delta_-) + \frac{\delta_-}{2}\phi'(\delta_-) $ with
\begin{align*}
\phi_\delta^*(r) =
\begin{cases}
\frac{\delta_-}{2\phi'(\delta_-)} r^2 - \phi(\delta_-) + \frac{\delta_-}{2}\phi'(\delta_-) &\text{for } r\in [0,\phi'(\delta_-)),\\
\phi^*(r) &\text{for } r\in [\phi'(\delta_-), \phi'(\delta_+)],\\
\frac{\delta_+}{2\phi'(\delta_+)}\,r^2 -\phi(\delta_+) +\frac{\delta_+}{2}\phi'(\delta_+)&\text{for } r\in (\phi'(\delta_+),\infty).
\end{cases}
\end{align*}
Since the shift in the dual integrand corresponds to the shift of the primal integrand, the duality results in Theorem~\ref{thm:Duality} still apply. We define the regularized energies
\begin{align*}
J_\delta(v) &\coloneqq \sum_{\alpha \in \cA} w_\alpha \phi_\delta(|(Bv)_\alpha|) - f \cdot v && \text{for all }v\in \mathbb{R}^N,\\
J^*_\delta(\tau)& \coloneqq \sum_{\alpha \in \cA} w_\alpha \phi_\delta^*(w_\alpha^{-1}|\tau_\alpha|) - \tau^\top Bg&&\text{for all }\tau \in \mathbb{R}^M.
\end{align*}
\begin{lemma}[Monotonicity]\label{lem:Monotonicity}
Suppose that $\phi$ satisfies the non-decreasing condition \eqref{eq:NonDecreasing}. Then the regularization is pointwise monotone in $\delta$ in the sense that for increasing intervals $\delta \subset \Delta \subset (0,\infty)$ and all $t\geq 0$
\begin{align*}
\phi_\delta(t) \leq \phi_\Delta(t)\qquad\text{and}\qquad \phi_\delta^*(t) \geq \phi^*_\Delta(t).  
\end{align*}
This yields for all $v\in \mathbb{R}^N$, $\tau \in \mathbb{R}^M$, and $\delta \subset \Delta \subset (0,\infty)$ the monotonicity
\begin{align*} 
J_\delta(v) \leq J_\Delta(v)\qquad \text{and}\qquad J_\delta^*(\tau) \geq J^*_\Delta(\tau).
\end{align*}
\end{lemma}
We skip the proof, as it follows by similar arguments as the special case $\Delta = (0,\infty)$ discussed in the following for the regularization errors defined for $t\geq 0$ by
\begin{align*}
e_\delta(t)\coloneqq \phi(t)-\phi_\delta(t)\qquad\text{and}\qquad e^*_\delta(t)\coloneqq \phi^*_\delta(t)-\phi^*(t).
\end{align*}
These errors equal
\begin{align}\label{eq:representationError}
\begin{aligned}
e_\delta(t) &= \begin{cases} 
\int^t_{\delta_-} s\big(\frac{\phi'(s)}{s}-\frac{\phi'(\delta_-)}{\delta_-}\big)\ds &\text{for } t\in [0, \delta_-),\\
0 &\text{for }t \in [ \delta_-,\delta_+],\\
\int_{\delta_+}^{t} s\big(\frac{\phi'(s)}{s}-\frac{\phi'(\delta_+)}{\delta_+}\big)\ds &\text{for } t\in ( \delta_+,\infty),
\end{cases}\\
e^*_\delta(t) &= \begin{cases} 
\int^t_{\phi'(\delta_-)} s\big(\frac{\delta_-}{\phi'(\delta_-)}-\frac{(\phi^*)'(s)}{s}\big)\ds &\text{for } t\in (0, \phi'(\delta_-)),\\
0 &\text{for }t \in [ \phi'(\delta_-),\phi'(\delta_+)],\\
\int_{\phi'(\delta_+)}^{t} s\big(\frac{\delta_+}{\phi'(\delta_+)}-\frac{(\phi^*)'(s)}{s} \big)\ds &\text{for } t\in ( \phi'(\delta_+),\infty).
\end{cases}
\end{aligned}
\end{align}
\begin{lemma}[Pointwise regularization error]\label{lem:PointwiseRegularizationError}
Suppose that $\phi$ satisfies the non-de\-creasing condition \eqref{eq:NonDecreasing}. Then we have
\begin{align*}
0 &\leq e_\delta(t) \leq  \begin{cases} 
 \frac12 \delta_- \phi'(\delta_-)-\phi(\delta_-) &\text{for } t\in [0, \delta_-),\\
0 &\text{for }t \in [ \delta_-,\delta_+],\\
\phi(t) - \phi(\delta_+)&\text{for } t\in ( \delta_+,\infty),
\end{cases}\\
0 &\leq e^*_\delta(t) \leq  \begin{cases} 
 \phi^*(\phi'(\delta_-))-\frac12 \delta_- \phi'(\delta_-) &\text{for } t\in [0, \phi'(\delta_-)),\\
0 &\text{for }t \in [ \phi'(\delta_-),\phi'(\delta_+)],\\
(t^2 - \phi'(\delta_+)^2)\frac{\delta_+}{2 \phi'(\delta_+)} &\text{for } t\in ( \phi'(\delta_+),\infty).
\end{cases}
\end{align*}
\end{lemma}
\begin{proof}
Let $t < \delta_-$. Then \eqref{eq:representationError} characterizes the error by
\begin{align*}
e_\delta(t) = \int^t_{\delta_-} s\Big(\frac{\phi'(s)}{s}-\frac{\phi'(\delta_-)}{\delta_-}\Big)\ds = \int_t^{\delta_-} s\Big(\frac{\phi'(\delta_-)}{\delta_-}- \frac{\phi'(s)}{s}\Big)\ds. 
\end{align*} 
Since $s \leq \delta_-$, the integrand is non-negative according to \eqref{eq:NonDecreasing}. The integral attains its maximal value at $t= 0$ and it reads 
\begin{align*}
\int_0^{\delta_-} s\Big(\frac{\phi'(\delta_-)}{\delta_-}- \frac{\phi'(s)}{s}\Big)\ds = \frac12 \delta_-\,\phi'(\delta_-)-\phi(\delta_-).
\end{align*}
If $\delta_+ < t$, the error reads 
\begin{align*}
\int_{\delta_+}^{t} s\Big(\frac{\phi'(s)}{s}-\frac{\phi'(\delta_+)}{\delta_+}\Big)\ds = \phi(t) - \phi(\delta_+) - \frac{\phi'(\delta_+)}{2\delta_+} |t^2 - \delta_+^2|.
\end{align*}
This concludes the proof of the first statement. The result for the dual integrand follows similarly in combination with Proposition~\ref{prop:Conjugate}.
\end{proof}
We can combine this upper bound for the pointwise regularization error with the following global error.
\begin{lemma}[Global regularization error]\label{lem:GlobReg}
Let $\delta = [\delta_-,\delta_+]\subset (0,\infty)$ be some relaxation interval and let
\begin{align*}
u = \argmin_{v\in V_0} J(v+g),\qquad u_\delta = \argmin_{v\in V_0} J_\delta(v+g),\qquad  \sigma = \argmin_{\tau \in \Sigma_f} J^*(\tau).
\end{align*}
Set $u_{\delta,g}\coloneqq u_\delta+g$ and $u_g\coloneqq u+g$. The impact of the regularization is controlled by
\begin{align*}
0 \leq J_\delta(u_g) - J_\delta(u_{\delta,g}) \leq \sum_{\alpha\in\cA} w_\alpha\,e^*_\delta(w_\alpha^{-1}|\sigma_\alpha|) - \sum_{\alpha\in\cA} w_\alpha\,e_\delta(|(Bu_g)_\alpha|).
\end{align*}
%
%
\end{lemma}
\begin{proof}
Let $\sigma_\delta = \argmin_{\tau \in \Sigma_f} J^*_\delta(\tau)$.
By duality (Theorem~\ref{thm:Duality}) and the definition of $e^*_\delta$ we have the bound
\begin{align*}
J(u_g) - J_\delta(u_{\delta,g}) = J_\delta^*(\sigma_\delta) - J^*(\sigma) \leq J_\delta^*(\sigma) - J^*(\sigma) = \sum_{\alpha\in\cA} w_\alpha\,e^*_\delta(w_\alpha^{-1}|\sigma_\alpha|).
\end{align*} 
Moreover, it holds that
\begin{align*}
J(u_g) = J_\delta(u_g) + \sum_{\alpha\in\cA} w_\alpha\,e_\delta(|(Bu_g)_\alpha|).
\end{align*}
Combining these statements concludes the proof.
%
\end{proof}
Since the contribution $\sum_{\alpha\in\cA} w_\alpha\,e^*_\delta(w^{-1}_\alpha |\sigma_\alpha|)$ tends to zero as the relaxation interval $\delta$ is enlarged, there exists for any $\varepsilon > 0$ some interval $\delta \subset (0,\infty)$ such that 
\begin{align*}
J_\delta(u_g)-J_\delta(u_{\delta,g})  \leq  \varepsilon.
\end{align*}
Using the notion of distance in Lemma~\ref{lem:natDist} and the convergence statements in Theorem~\ref{thm:ConvOfDualIRLS} and ~\ref{thm:ConvPrimal} we thus obtain for this relaxation interval $\delta$, the iterates $u_g^n \in \mathbb{R}^N$ in our dIRLS scheme \eqref{eq:DualKac}, and the minimizer of the regularized problem $\sigma_\delta \in \Sigma_f$ with constants $ \kappa_\delta \in (0,1)$ and $C_\delta<\infty$ depending on $\delta = \delta(\varepsilon)$ the bound
\begin{align}\label{eq:COnvergenceResult}
\begin{aligned}
&\sum_{\alpha \in \cA} w_\alpha |V_{\phi_\delta}((B u_g)_\alpha) - V_{\phi_\delta}((B u_g^n)_\alpha)|^2\\
&\qquad\leq 2 \sum_{\alpha \in \cA} w_\alpha |V_{\phi_\delta}((B u_g)_\alpha) - V_{\phi_\delta}((B u_{\delta,g})_\alpha)|^2  \\
&\qquad\quad + 2\sum_{\alpha \in \cA} w_\alpha |V_{\phi_\delta}((Bu^n_g)_\alpha) - V_{\phi_\delta}(B (u_{\delta,g})_\alpha)|^2 \\
&\qquad \lesssim J_\delta (u_g) - J_\delta(u_{\delta,g}) + J_\delta (u^n_g) - J_\delta(u_{\delta,g})\\
&\qquad \leq \varepsilon + C_\delta (1-\kappa_\delta)^n\big(J^*_\delta(\sigma^0) - J_\delta^*(\sigma_\delta) \big).
\end{aligned}
\end{align}
\section{Applications}
We conclude the paper with a numerical study of the dIRLS scheme in MATLAB. 
The implementations used to generate the results below can be found in \cite{StornCode}. We solve linear systems of equations using MATLAB's direct solver \texttt{mldivide}.  

\subsection{Linear regression in $\ell^p$}\label{subsec:ExpLinRegresssion}
In our first experiment we investigate the performance of our dIRLS scheme for the (unconstrained) linear $\ell^p$ regression discussed in Section~\ref{subsec:linReg}; that is, given a matrix $A \in \mathbb{R}^{M\times N}$ with $M\geq N$, a vector $b \in \mathbb{R}^M$, and an exponent $p \in [2,\infty)$ we seek the minimizer 
\begin{align*}
u = \argmin_{v\in \mathbb{R}^N}\, \lVert A v - b\rVert_{\ell^p}^p.
\end{align*}
We compare our method numerically to a Newton scheme with regularized integrand $\phi_\textup{Newton}(t) \coloneqq (t^2 + \epsilon^2)^{p/2}$ and line-search using backtracking. Furthermore, we compare it to the $p$-IRLS scheme -- a modification of the IRLS scheme for exponents $p> 2$ introduced in \cite{AdilPengSachdeva19} (which should perform similarly to the modified IRLS scheme in \cite{AdilKyngPengSachdeva24} as claimed in \cite[Sec.~6.2]{AdilKyngPengSachdeva24}). 
%
This scheme performs better than the MATLAB/CVX solver and homotopy-based approaches \cite[Sec.~4]{AdilKyngPengSachdeva24}. The method has been refined in \cite{EneNguyenVladu25}.
According to their numerical experiments for $p=8$, this refinement requires fewer iterations in the sense that they need approximately $80\%$ of the iterations in the $p$-IRLS to obtain results of similar quality.
Our experiment follows \cite{AdilPengSachdeva19,AdilKyngPengSachdeva24}: We create a matrix $A \in \mathbb{R}^{M\times N}$ and a vector $b\in \mathbb{R}^M$ with $4500 = N < M = 5000$. The entries are sampled from the uniform distribution on $(0,1)$ in the sense that they are for all $i = 1,\dots,M$ and $j=1,\dots,N$ independent and uniformly distributed random variables
\begin{align*}
A_{ij} \sim \mathcal{U}(0,1)\qquad\text{and}\qquad b_i \sim \mathcal{U}(0,1).
\end{align*}
We seek with fixed exponent $p\in [2,\infty)$ the unconstrained minimizer
\begin{align*}
u = \argmin_{v \in \mathbb{R}^N}\, \lVert A v - b \rVert_{\ell^p}^p.
\end{align*}
A reformulation of the problem fitting into our framework is discussed in Section~\ref{subsec:linReg}. We choose the regularization discussed in Section~\ref{subsec:Regularization}; that is, with some interval $\delta = [\delta_-,\delta_+] \subset (0,\infty)$ we set 
\begin{align*}
\phi_\delta'(t) \coloneqq (\delta_- \vee t \wedge \delta_+)^{p-2} t \qquad \text{for all }t\geq 0.
\end{align*}
Given $\sigma^n \in \mathbb{R}^M$, our resulting routine then seeks in each iteration the vector $u^{n+1} \in \mathbb{R}^N$ that solves with weights $a_\alpha^n \coloneqq  |\sigma^n_\alpha|/(\phi_\delta^*)'(|\sigma^n_\alpha|)$ the variational problem
\begin{align*}
\sum_{\alpha = 1}^M a^n_\alpha (A{u}^{n+1} - b )_\alpha (Av)_\alpha = 0\qquad\text{for all }v \in \mathbb{R}^N.
\end{align*}
With diagonal matrix $D^n \coloneqq \textup{diag}(a^n)\in \mathbb{R}^{M \times M}$ this reads equivalently
\begin{align*}
A^\top D^n A u^{n+1} = A^\top D^n b.
\end{align*}
After each solve, we update the dual variable by $\sigma^{n+1} \coloneqq D^n (A u^{n+1}- b)$.

Figure~\ref{fig:Exp1a} displays the convergence history for exponents $p\in \lbrace 10,20,40,80\rbrace$ and relaxation interval $\delta = [10^{-9},10^9]$. As expected, for moderate values of $p$ the Newton scheme with line search converges much faster than the dIRLS and $p$-IRLS scheme. However, for larger exponents $p$, we observe a pre-asymptotic regime with slow convergence. For even larger exponents $p$ the scheme does not converge. The $p$-IRLS scheme converges for a wider regime of exponents $p$. However, for $p = 80$ its practical limitations are reached as well -- even though the theory states convergence for any exponent $p>2$, the large exponent $p$ causes numerical instabilities that result in unfeasible line searches. Our dIRLS scheme converges for any tested value of $p$ (even though the rate of convergence deteriorates as $p$ is enlarged). For $p = 20,40,80$ it furthermore converges much faster than the other schemes. 
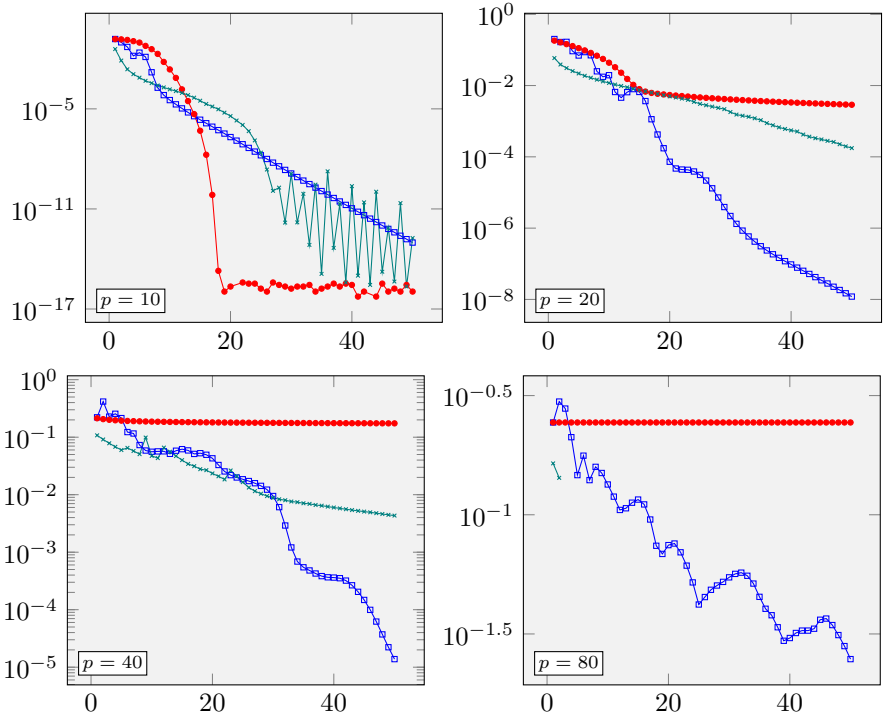
\begin{figure}
\begin{tikzpicture}
\begin{axis}[
clip=false,
width=.5\textwidth,
height=.45\textwidth,
cycle multi list={\nextlist MyColors3},
scale = {1},
ymode = log,
clip = true,
legend cell align=left,
legend style={legend columns=1,legend pos = north east,font=\fontsize{7}{5}\selectfont},
]
	\addplot table [x=iter,y=dIRLS] {Experiments/Exp1a_p10.dat};
	\addplot table [x=iter,y=Newton] {Experiments/Exp1a_p10.dat};
	\addplot table [x=iter,y=pIRLS] {Experiments/Exp1a_p10.dat};	
	\node[
  anchor=south west,
  draw=black,
  fill=white,
  font=\fontsize{7}{5}\selectfont,
  inner sep=1.5pt
] at (rel axis cs:0.03,0.03) {$p=10$};
\end{axis}
\end{tikzpicture}
\begin{tikzpicture}
\begin{axis}[
clip=false,
width=.5\textwidth,
height=.45\textwidth,
cycle multi list={\nextlist MyColors3},
scale = {1},
ymode = log,
clip = true,
legend cell align=left,
legend style={legend columns=1,legend pos = north east,font=\fontsize{7}{5}\selectfont},
]
	\addplot table [x=iter,y=dIRLS] {Experiments/Exp1a_p20.dat};
	\addplot table [x=iter,y=Newton] {Experiments/Exp1a_p20.dat};
	\addplot table [x=iter,y=pIRLS] {Experiments/Exp1a_p20.dat};	
	\node[
  anchor=south west,
  draw=black,
  fill=white,
  font=\fontsize{7}{5}\selectfont,
  inner sep=1.5pt
] at (rel axis cs:0.03,0.03) {$p=20$};
\end{axis}
\end{tikzpicture}
\begin{tikzpicture}
\begin{axis}[
clip=false,
width=.5\textwidth,
height=.45\textwidth,
cycle multi list={\nextlist MyColors3},
scale = {1},
ymode = log,
clip = true,
legend cell align=left,
legend style={legend columns=1,legend pos = north east,font=\fontsize{7}{5}\selectfont},
]
	\addplot table [x=iter,y=dIRLS] {Experiments/Exp1a_p40.dat};
	\addplot table [x=iter,y=Newton] {Experiments/Exp1a_p40.dat};
	\addplot table [x=iter,y=pIRLS] {Experiments/Exp1a_p40.dat};	
		\node[
  anchor=south west,
  draw=black,
  fill=white,
  font=\fontsize{7}{5}\selectfont,
  inner sep=1.5pt
] at (rel axis cs:0.03,0.03) {$p=40$};
\end{axis}
\end{tikzpicture}
\begin{tikzpicture}
\begin{axis}[
clip=false,
width=.5\textwidth,
height=.45\textwidth,
cycle multi list={\nextlist MyColors3},
scale = {1},
ymode = log,
clip = true,
legend cell align=left,
legend style={legend columns=1,legend pos = north east,font=\fontsize{7}{5}\selectfont},
]
	\addplot table [x=iter,y=dIRLS] {Experiments/Exp1a_p80.dat};\label{line1}
	\addplot table [x=iter,y=Newton] {Experiments/Exp1a_p80.dat};\label{line2}
	\addplot table [x=iter,y=pIRLS] {Experiments/Exp1a_p80.dat};\label{line3}
		\node[
  anchor=south west,
  draw=black,
  fill=white,
  font=\fontsize{7}{5}\selectfont,
  inner sep=1.5pt
] at (rel axis cs:0.03,0.03) {$p=80$};
\end{axis}
\end{tikzpicture}
\caption{Distance $\lVert Au^n - b \rVert_{\ell^p} - \lVert Au - b \rVert_{\ell^p}$ plotted against the number of iterations $n$ for the dIRLS scheme \ref{line1}, Newton's method \ref{line2}, and the $p$-IRLS scheme \ref{line3}.}\label{fig:Exp1a}
\end{figure}

\subsection{Graph $p$-Laplacian}\label{subsec:GraphPlap}
In this experiment we investigate the beneficial properties of the dIRLS scheme for problems where the graph $p$-Laplacian serves as a smoother, cf.~Section~\ref{subsec:app-ssl}. 
Notice that such applications offer the beneficial property that by design they should avoid large gradients, which reduces or removes the impact of the upper regularization bound $\delta_+$ even for moderate values of $\delta_+$. Indeed, in our experiments enlarging the relaxation interval $\delta = [10^{-3},10^3]$ led to the same results.  
We recreate an experiment from \cite{Flores18,FloresCalderLerman22}, where we use a scattering transform \cite{BrunaMallat13} to extract
features from a dataset which are then used to build the symmetric $K$-nearest neighbor graph with $K=10$ and weights 
\begin{align*}
w_e \coloneqq \exp(-|i-j|^2/\sigma^2)\text{ for } e = [i,j] \in \mathcal{E} \text{ and } \sigma\coloneqq \max\lbrace |i-j|/2 \colon e = [i,j]\in \mathcal{E}\rbrace. 
\end{align*}
Our first dataset is the MNIST dataset \cite{LecunBottouBengioHaffner98} containing $N = 70\, 000$ grayscale $28 \times 28$ pixel images of handwritten digits from $0$ to $9$, with about $7000$ images for each digit.
The second dataset is the Fashion MNIST \cite{XiaoRasulVollgrad17} which is similar to the MNIST except ten fashion items replacing the digits.
We aim at classifying the items (digits/fashion items); that is, given one labeled image for each item, we solve the $10$-class classification problem by a standard one-vs-rest approach. More precisely, for a fixed item, we assign to the corresponding labeled image the boundary value $+1$, while each labeled image representing another item obtains the boundary value $-1$. We then solve the graph $p$-Laplace problem, resulting in a function $u = (u_j)_{j=1}^N \in \mathbb{R}^N$ that assigns to each vertex $j$ the score $u_j$. Performing this computation for each item gives scores, and the final classification is chosen as the class with the highest score.
We fix the exponent $p=10$ in our computations. The results for larger exponents are similar. Our first iterate  $\sigma^1 \in \mathbb{R}^M$ results from solving the graph Laplacian.

In \cite{FloresCalderLerman22} the problem is solved by a Newton method with homotopy in $p$. This worked sufficiently well for their applications, but it requires some fine-tuning of parameters. Other approaches include the MATLAB/CVX solver and the $p$-IRLS schemes of \cite{AdilKyngPengSachdeva24,EneNguyenVladu25}, which were outperformed by our dIRLS method. Our results for the problem at hand are displayed in Figure~\ref{fig:Exp2a}. It shows that the first dIRLS iteration (using the solution to the graph Laplacian as the initial iterate) results in a significant gain with respect to the accuracy of our predictions and a significant improvement of the energy approximation. Thereafter, the dIRLS proceeds to converge linearly towards the minimizer. However, this does not improve the accuracy of the prediction anymore. In fact, in machine learning applications it often suffices to obtain a decent approximation of the minimizer. In our experiments, such decent approximations were obtained with one additional dIRLS iteration after solving the graph Laplacian. In other words, with roughly twice the effort of a standard graph-Laplacian solve we could improve the accuracy of the semi-supervised learning routine significantly. This makes the dIRLS a very attractive method for these types of applications. 

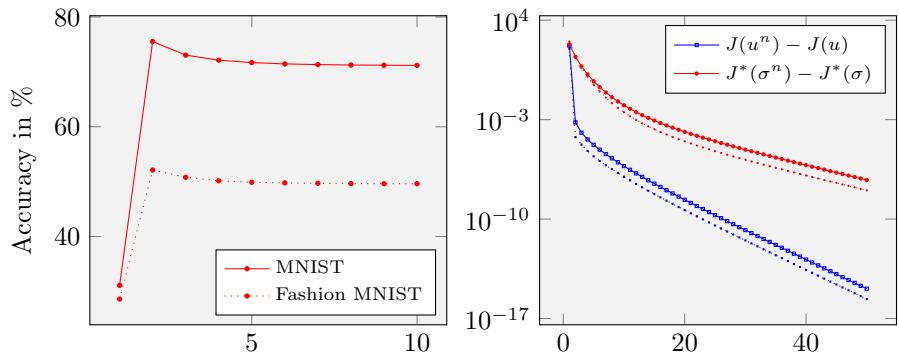
\begin{figure}
{\centering
\begin{tikzpicture}
\begin{axis}[
clip=false,
width=.5\textwidth,
height=.45\textwidth,
cycle multi list={\nextlist MyColors},
scale = {1},
ylabel = {Accuracy in $\%$},
clip = true,
legend cell align=left,
legend style={legend columns=1,legend pos= south east,font=\fontsize{7}{5}\selectfont}
]
	\addplot table [x=iter,y=CorrectInPercent] {Experiments/Exp2a_70000_p_10.dat};
	\addplot table [x=iter,y=CorrectInPercent] {Experiments/Exp2b_70000_p_10.dat};	
	\legend{{MNIST},{Fashion MNIST}};

\end{axis}
\end{tikzpicture}
\begin{tikzpicture}
\begin{axis}[
clip=false,
width=.5\textwidth,
height=.45\textwidth,
cycle multi list={\nextlist MyColors2},
scale = {1},
ymode = log,
clip = true,
legend cell align=left,
legend style={legend columns=1,legend pos = north east,font=\fontsize{7}{5}\selectfont},
]
	\addplot table [x=iter,y=err_J] {Experiments/Exp2a_70000_p_10_Conv.dat};
	\addplot table [x=iter,y=err_J_star] {Experiments/Exp2a_70000_p_10_Conv.dat};
	\addplot table [x=iter,y=err_J] {Experiments/Exp2b_70000_p_10_Conv.dat};
	\addplot table [x=iter,y=err_J_star] {Experiments/Exp2b_70000_p_10_Conv.dat};
	\legend{{$J(u^n) - J(u)$},{$J^*(\sigma^n) - J^*(\sigma)$}};
\end{axis}
\end{tikzpicture}
\caption{The left-hand side displays the accuracy in predicting the $70\, 000$ digits/fashion items. The right-hand side displays the energy errors for the MNIST (thick line) and the Fashion MNIST (thin line).} \label{fig:Exp2a}}
\end{figure}
\subsection*{Acknowledgments}
The author is grateful to Professor Jeffrey Calder for insightful discussions on the importance of the graph $p$-Laplacian in semi-supervised learning which initiated the research presented here.

\printbibliography

\end{document}